\documentclass[10pt]{amsart}

\usepackage{amssymb,amsfonts}
\usepackage[all]{xy}
\usepackage[shortlabels]{enumitem}

\newcommand{\Z}{\ensuremath{\mathbb{Z}}}

\newcommand{\field}{\Bbbk}

\DeclareMathOperator{\Aut}{Aut}

\DeclareMathOperator{\op}{op}
\DeclareMathOperator{\supp}{supp}

\theoremstyle{plain}
\newtheorem{theorem}{Theorem}[section]
\newtheorem{proposition}[theorem]{Proposition}
\newtheorem{lemma}[theorem]{Lemma}
\newtheorem{corollary}[theorem]{Corollary}

\theoremstyle{remark}

\newtheorem{remark}[theorem]{Remark}
\newtheorem{problem}[theorem]{Problem}

\begin{document}

\title[Free group algebras in division rings]{Free symmetric group algebras in division rings 
generated by poly-orderable groups}

\author[V. O. Ferreira]{Vitor O. Ferreira}
\address{Department of Mathematics - IME, University of S\~ao Paulo,
Caixa Postal 66281, S\~ao Paulo, SP, 05314-970, Brazil}
\email{vofer@ime.usp.br}
\thanks{The first and second authors were partially supported by
Fapesp-Brazil Proc.~2009/52665-0.}

\author[J. Z. Gon\c calves]{Jairo Z. Gon\c calves}
\address{Department of Mathematics - IME, University of S\~ao Paulo,
Caixa Postal 66281, S\~ao Paulo, SP, 05314-970, Brazil}
\email{jz.goncalves@usp.br}
\thanks{The second author was partially supported by CNPq-Brazil Grant 300.128/2008-8.}

\author[J. S\'{a}nchez]{Javier S\'{a}nchez}
\address{Department of Mathematics - IME, University of S\~ao Paulo,
Caixa Postal 66281, S\~ao Paulo, SP, 05314-970, Brazil}
\email{jsanchez@ime.usp.br}
\thanks{The third author was supported by Fapesp-Brazil Proc.~2009/50886-0.}



\date{26 June 2012}

\begin{abstract}
  We show that the canonical involution on a nonabelian
  poly-orderable group $G$ extends to the Hughes-free 
  division ring of fractions $D$ of the group algebra $\field[G]$
  of $G$ over a field $\field$ and that, with respect
  to this involution, $D$ contains a pair of symmetric elements
  freely generating a free group subalgebra of $D$ over $\field$.
\end{abstract}

\maketitle

\section*{Introduction}

A long-standing conjecture of Makar-Limanov states that a division
ring which is finitely generated (as a division ring) and infinite-dimensional 
over its center contains a free algebra of rank $2$ over the center \cite{lM84}. 

This question has been extensively investigated, see \textit{e.g.}~\cite{GS12}, and 
recently an interesting advance appeared in the context of division rings generated
by a group algebra of an ordered group; namely, in \cite{jSpp} S\'anchez proved that
if $K$ is a division ring, $G$ is a nonabelian ordered group and $K(G)$ is
the subdivision ring of the Malcev-Neumann series division ring $K((G))$ generated
by the group ring $K[G]$, then $K(G)$ contains a free group algebra of rank $2$ over
its center.

\medskip

In the presence of an involution, a natural question would be whether a division
ring satisfying Makar-Limanov's conjecture contains a free algebra of rank $2$
generated by symmetric elements. 

An analogous problem was considered in \cite{GS12}, 
where the authors looked for a symmetric pair generating a free subgroup in
the multiplicative group of a division ring endowed with an involution.

\medskip

In order to state the main contribution of the present paper, we shall introduce
some definitions and notation.

In what follows, $\field$ will denote a field. Given a $\field$-algebra $A$, by 
a \emph{$\field$-involution} on $A$ one understands a $\field$-linear map 
${}^{\star}\colon A\to A$ satisfying
\begin{enumerate}[(i)]
\item $(ab)^{\star}=b^{\star}a^{\star}$, for all $a,b\in A$, and
\item $(a^{\star})^{\star}=a$, for all $a\in A$.
\end{enumerate}
An element $a\in A$ is said to be \emph{symmetric} if $a^{\star}=a$.

If $G$ is a group and $\field[G]$ denotes the group algebra of $G$ over
$\field$, then the map
$$
\begin{array}{rcl}
\field[G] & \longrightarrow & \field[G]\\
\sum_{x\in G}xa_x & \longmapsto & \sum_{x\in G}x^{-1}a_x,
\end{array}
$$
where for all $x\in G$, $a_x$ are elements of $\field$ all but
a finite number of which nonzero, is a $\field$-involution $\field[G]$, henceforth 
called the \emph{canonical involution} of $\field[G]$.

If $G$ is ordered, we let $\field((G))$ denote the division ring of Malcev-Neumann 
series of $G$ over $\field$, and let $\field(G)$ be the subdivision ring of
$\field((G))$ generated by $\field[G]$. The main aim of this paper is to present a 
proof that if $G$ is nonabelian, then the canonical involution on $\field[G]$ extends to a $\field$-involution ${}^{\star}$ on $\field(G)$ and $\field(G)$ contains 
a group $\field$-algebra of a nonabelian free group generated by symmetric elements 
with respect to ${}^{\star}$. Moreover, a pair of free symmetric generators
is explicitly constructed.

\medskip

In Section~\ref{sec:nilp}, we approach the case of a torsion-free nilpotent group of 
class $2$ generated by $2$ elements. This case is treated separately for two reasons. 
First, it is a step in the proof of the general case and, second,
it has interest in itself for, in this case, the group algebra
$\field[G]$ is a noetherian domain and, therefore, has a unique
field of fractions, regardless of it being embeddable in a
Malcev-Neumann series ring.

Section~\ref{sec:gen-ext} is devoted to a proof of the fact that
the canonical involution on $\field[G]$ can be extended to an involution on
$\field(G)$ for an arbitrary ordered group. To be more precise,
we consider the larger class of locally indicable groups which
have a crossed product with a Hughes-free division ring of fractions.

The existence of free symmetric generators of a free group algebra
inside $\field(G)$ for $G$ orderable is tackled in Section~\ref{sec:gen-free}.
The proof itself splits in two cases, depending on properties
of the ordering of $G$, according to the trichotomy described
in \cite{jSpp}. Here, again, we consider the more general
case of crossed products of locally indicable groups with a
Hughes-free division ring of fractions. The main result in
the paper is Theorem~\ref{teo:freegen} which states that a
class of groups including poly-orderable groups satisfies
Makar-Limanov's conjecture with respect to free group algebras
generated by symmetric elements.

Finally, in Section~\ref{sec:prob} we propose some questions and indicate possible 
developments springing from the ideas in the paper.

\medskip

In what follows, if $D$ is a division ring, which might be a field, we shall let
$D^{\times}$ denote the set of nonzero elements of $D$.

\section{The nilpotent case}\label{sec:nilp}

Let $G$ be the free nilpotent group of class $2$ generated by two elements, that
is, 
$$
G = \langle x,y : [[x,y],x]=[[x,y],y]=1\rangle,
$$
where $[x,y]=x^{-1}y^{-1}xy$. Let $\field$ be a field and let ${}^{\star}$ denote the
canonical $\field$-involution on the group algebra $\field[G]$ of $G$ over $\field$.

It is well known that $\field[G]$ is a noetherian domain and,
thus, ${}^{\star}$ extends to a $\field$-involution (still denoted
by ${}^{\star}$) on the Ore field of fractions $D$ of $\field[G]$.

Our aim in this section is to provide a proof of the following
fact.

\begin{theorem}\label{teo:1}
  Let $\field$ be a field and consider the group
  $G = \langle x,y : [[x,y],x]=[[x,y],y]=1\rangle$.
  Let $D$ denote the Ore field of fractions of the group algebra $\field[G]$.
  Then 
  $$
  1+y(1-y)^{-2}
  \quad\text{and}\quad
  1+ y(1-y)^{-2}x(1-x)^{-2}y(1-y)^{-2}
  $$
  are symmetric elements with respect to the canonical involution on $D$ and 
  freely generate a free group $\field$-algebra in $D$.
\end{theorem}

Let $\field(\lambda)$ denote the field of rational functions in the
indeterminate $\lambda$ over $\field$ and let $K=\field(\lambda)(t)$ be
the field of rational functions in the indeterminate $t$ over $\field(\lambda)$.
It is known that $D$ is $\field$-isomorphic to the Ore field of fractions $Q$ of the skew 
polynomial ring $K[X;\sigma]=\bigl\{\sum_{i\geq 0}X^ia_i : a_i\in K \text{ all but a finite
number nonzero}\bigr\}$, 
where $\sigma$ is the $\field(\lambda)$-automorphism of $K$ satisfying 
$t^{\sigma}=\lambda t$ and $aX=Xa^{\sigma}$, for all $a\in K$. An explicit isomorphism
$Q\to D$ is given by $\lambda\mapsto [x,y], t\mapsto x, X\mapsto y$.

The main ingredient in the proof of Theorem~\ref{teo:1} is the following
proposition, whose proof is given below, after the proof of
Theorem~\ref{teo:1}.

\begin{proposition}\label{prop:2}
  Let $\field$ be a field and let $Q$ denote the Ore field of fractions of
  the skew polynomial ring $K[X;\sigma]$, where $K=\field(\lambda)(t)$ and
  $\sigma$ is the $\field(\lambda)$-automorphism of $K$ such that 
  $t^{\sigma}=\lambda t$. Then the elements $(1-X)^{-1}$ and
  $t(1-t)^{-2}(1-X)^{-1}$ freely generate a free $\field$-subalgebra of $Q$.
\end{proposition}

\begin{proof}[Proof of Theorem~\ref{teo:1}]
  We work in the field of fractions $Q$ of the skew polynomial ring
  $K[X;\sigma]$, with the notation introduced in the paragraph preceding
  the statement of Proposition~\ref{prop:2}. In particular, the 
  $\field$-involution ${}^{\star}$ of $D$ extending the canonical involution
  of $\field[G]$ defines a $\field$-involution ${}^{\star}$ on $Q$ satisfying
  $\lambda^{\star}=\lambda^{-1}$, $t^{\star}=t^{-1}$ and $X^{\star}=X^{-1}$.
  
  Set $\alpha = (1-X)^{-1}$ and $\beta = t(1-t)^{-2}$.
  By Proposition~\ref{prop:2}, $\alpha$ and $\beta\alpha$ freely generate a free
  $\field$-subalgebra $A$ of $Q$. 
  Clearly, $\alpha^2-\alpha, \beta(\alpha^2-\alpha)\in A$ and they
  do not commute. It then follows from \cite[Corollary~6.7.4]{pC06} that $\alpha^2-\alpha$
  and $\beta(\alpha^2-\alpha)$ freely generate a free $\field$-subalgebra of $A$ and,
  thus, of $Q$.
  Therefore, $\gamma=\alpha^2-\alpha$ and $\delta=(\alpha^2-\alpha)\beta(\alpha^2-\alpha)$ freely
  generate a free $\field$-subalgebra of $Q$ too.
  But $\gamma = X(1-X)^{-2}$ and so $\gamma^{\star}=\gamma$.
  Also $\beta^{\star}=\beta$ and this implies that $\gamma,\delta$
  form a pair of symmetric free generators
  of a free $\field$-subalgebra of $Q$. Now consider the $X$-adic valuation
  on $Q$, that is, the valuation $\nu$ in $Q$ which is zero in $K$ and
  $\nu(X)=1$. We have $\nu(\gamma)=1$ and $\nu(\delta)=2$.
  By \cite[Corollary~1 on p.~524]{aL84}, $1+\gamma$ and $1+\delta$
  is a pair of symmetric elements in $Q$ which freely generates a
  free group $\field$-algebra in $Q$.
  
  Finally, one uses the isomorphism between $Q$ and $D$ to obtain
  the desired result.
\end{proof}

We now proceed to a proof of Proposition~\ref{prop:2} using
a criterion of Bell and Rogalski \cite{BRpp} for a pair of elements
to generate a free subalgebra in the field of fractions of a
skew polynomial ring.

Let $k$ be a field and let $F$ be an extension field of $k$.
Given $\sigma\in\Aut_kF$,
let $F_{\sigma}=\{f\in F : f^{\sigma}=f\}$ denote the fixed subfield
of $F$ with respect to $\sigma$. We shall make use of the following
simplified form of the result of Bell and Rogalski.

\begin{theorem}[{\cite[Theorem~2.2]{BRpp}}]\label{teo:bell}
  Let $F/k$ be a field extension and let $\sigma\in\Aut_kF$.
  Denote by $Q$ the Ore field of fractions of the skew
  polynomial ring $F[X;\sigma]$.
  Let $b\in F\setminus F_{\sigma}$ be such that for all $f\in F$,
  $f^{\sigma}-f\in  F_{\sigma}+F_{\sigma}b$ implies $f\in F_{\sigma}$.
  Then the elements $(1-X)^{-1}$ and $b(1-X)^{-1}$ freely generate
  a free $k$-subalgebra of $Q$. \qed
\end{theorem}

We shall apply the above criterion to the case where $F$ is a field
of rational functions over $k$ and in order to verify the hypothesis
of Theorem~\ref{teo:bell}, we shall use the notion of a pole of a
rational function. So let $k(t)$ denote the field of rational functions
over $k$ on the indeterminate $t$ and let $\bar{k}$ denote the
algebraic closure of $k$. We add an element $\infty$ to $\bar{k}$ and
regard the elements of $k(t)$ as functions on $\bar{k}\cup\{\infty\}$ with 
the following conventions: given $h\in k(t)$, say $h=\frac{p}{q}$, with
$p,q\in k[t]$ and $\gcd(p,q)=1$, we set $h(r) = \infty$, for all $r\in\bar{k}$ 
such that $q(r)=0$, and 
$$
h(\infty) = \begin{cases}
  \infty & \text{if $\deg(p)>\deg(q)$,}\\
  0 & \text{if $\deg(p)<\deg(q)$,}\\
  ab^{-1} & \text{if $\deg(p)=\deg(q)$,}
  \end{cases}
$$
where $a$ stands for the leading term of $p$ and $b$ for the
leading term of $q$. We say that an element $r\in\bar{k}\cup\{\infty\}$ is a pole
of $h\in k(t)$ if $h(r)=\infty$. Given $h\in k(t)$ we shall let
$S_h$ denote the set of poles of $h$. Then, clearly, $S_h$ is a
finite set which is nonempty whenever $h\not\in k$. Also,
$h\in k[t]\setminus k$ if and only if $S_h=\{\infty\}$. Note,
finally, that if $\sigma\in\Aut_kk(t)$ and $t^{\sigma}=h$, then for all $f\in k(t)$,
we have $S_f=h(S_{f^\sigma})$.

The following lemma is based on an argument of Lorenz (see \cite[Lemma~1]{mL86}).

\begin{lemma}\label{le:lorenz}
  Let $k$ be a field, let $\lambda\in k^{\times}$ be an element which is 
  not a root of unity and let $\sigma\in\Aut_kk(t)$ be the automorphism
  such that $t^{\sigma}=\lambda t$. Let $g\in k(t)\setminus k[t]$ be 
  a rational function which has a unique pole and this pole is nonzero. 
  If $f\in k(t)$ satisfies $f^{\sigma}-f\in k+kg$, then $f\in k$.
\end{lemma}

\begin{proof}
  We start by remarking that, taking Laurent series representing
  the rational functions, for all $u\in k(t)$, $u^{\sigma}-u\in k$ implies
  that $u\in k$. So it is enough to show that $f^{\sigma}-f\in k$.

  Let $A=S_{f^{\sigma}}$ and $B=S_f$. If $t^{\sigma}=h$, then, as we 
  have seen, $B=h(A)$. Let $D=(A\cup B)\setminus (A\cap B)$ and
  let $C=S_{f^{\sigma}-f}$. Then
  $D$ is a subset of $C$ with an even number of elements, for
  $|A| = |B|$. Suppose that $f^{\sigma}-f= a+bg$ for
  some $a,b\in k$. Then $C=S_{a+bg}\subseteq S_g$. But, by hypothesis,
  $S_g$ contains just one element; it follows that $D=\emptyset$ and $B=A$.
  Thus, $B=h(B)=h^2(B)=\dots$. 
  If there existed $r\in B$, with $r\ne 0$ and 
  $r\ne \infty$, we would have $\{r,h(r),h^2(r),\dotsc\}\subseteq B$.
  Since $B$ is finite, this would imply that $\lambda^mr=h^m(r)=h^n(r)=\lambda^nr$, 
  for some $0\leq m<n$, which is impossible, because $\lambda$ is not a root
  of $1$. If either $B=\{0\}$ or $B=\{0,\infty\}$, we would have that
  $f=p/t^n$, for some $n\geq 1$ and some $p\in k[t]$ which is not a multiple
  of $t$. In this case, $f^{\sigma}-f=\big(\lambda^{-n}p^{\sigma}-p\big)/t^n$.
  Now the numerator $\lambda^{-n}p^{\sigma}-p$ does not vanish, because
  $t$ does not divide $p$ and $\lambda^n\ne 1$. So, $0$ would be a pole
  of $f^{\sigma}-f$, that is, we would have $0\in C\subseteq S_g$ and this
  would contradict the hypothesis on $g$.
  We are left, then, with two possibilities: either $B=\emptyset$ or $B=\{\infty\}$.
  In either case, $a+bg=f^{\sigma}-f\in k[t]$, and this can only be possible
  if $b=0$, for, by hypothesis, $g\not\in k[t]$. So $f^{\sigma}-f\in k$, which
  is what we needed to show.  
\end{proof}

\begin{theorem}
  Let $k$ be a field, let $\lambda\in k^{\times}$ be an element which is 
  not a root of unity and let $\sigma\in\Aut_kk(t)$ be the automorphism
  such that $t^{\sigma}=\lambda t$. Let $Q$ denote the Ore field of
  fractions of the skew polynomial ring $k(t)[X;\sigma]$. Then the
  elements $(1-X)^{-1}$ and $t(1-t)^{-2}(1-X)^{-1}$ freely generate a 
  free $k$-subalgebra in $Q$.
\end{theorem}

\begin{proof}
  We shall apply the criterion in Theorem~\ref{teo:bell}. So let $F=k(t)$.
  We have seen in the first paragraph of the proof of Lemma~\ref{le:lorenz}
  that, for all $u\in F$, $u^{\sigma}-u\in k$ implies $u\in k$. In particular,
  it follows that $F_{\sigma}=k$. The element $b=t(1-t)^{-2}\in F$ is a rational
  function which is not a polynomial and has a unique pole which is nonzero. So,
  by Lemma~\ref{le:lorenz}, the hypothesis of Theorem~\ref{teo:bell} are satisfied
  and $(1-X)^{-1}$ and $t(1-t)^{-2}(1-X)^{-1}$ freely generate a 
  free $k$-subalgebra in $Q$.
\end{proof}

To get Proposition~\ref{prop:2} from the theorem above, let $k=\field(\lambda)$
and note that since $(1-X)^{-1}$ and $t(1-t)^{-2}(1-X)^{-1}$ freely generate a 
free $k$-subalgebra in $Q$, they freely generate a free $\field$-subalgebra in $Q$.

We end this section with a direct consequence of Theorem~\ref{teo:1}.

\begin{corollary}
  Let $\field$ be a field and let $G$ be a nonabelian torsion-free nilpotent
  group. Then the Ore field of fractions of the group algebra $\field[G]$
  contains a free group $\field$-algebra of rank $2$ freely generated by
  symmetric elements with respect to the canonical involution.
\end{corollary}

\begin{proof}
  Since $\field[G]$ is a noetherian domain, the Ore field of fractions $Q$
  of $\field[G]$ exists and the canonical involution of $\field[G]$ extends
  to a $\field$-involution on $Q$. Since $G$ is nilpotent and nonabelian, there
  exist $a,b\in G$ such that $[a,b]=c\ne 1$ and $[a,c]=[b,c]=1$. Then the
  subgroup $H$ of $G$ generated by $a$ and $b$ is a free nilpotent group
  of class $2$ generated by $2$ elements. By Theorem~\ref{teo:1}, the subdivision
  ring $D$ of $Q$ generated by $\field[N]$ contains a pair of symmetric elements
  with respect to the canonical involution on $\field[N]$, and, thus, with
  respect to the canonical involution on $\field[G]$, which freely generate
  a free group $\field$-algebra in $D\subseteq Q$.
\end{proof}

\section{The general case: extending the canonical involution}
\label{sec:gen-ext}

In this section we shall consider more general groups.
More precisely, the main result in the section,
Theorem~\ref{teo:extension}, will show that an involution
on a crossed product of a locally indicable group
with a Hughes-free division ring of fractions can
be uniquely extended to this division ring.

We must start by recalling the necessary definitions,
notation and basic results.
\subsection{General crossed products}

Given a group $G$ and a ring $R$, a \emph{crossed product}
of $G$ over $R$ is an associative ring $R\ast G$ containing
$R$ as a subring and such that $R\ast G$ is a free right
$R$-module with basis $\overline{G}=\{\bar{x} : x\in G\}$, a set in bijection
with $G$, with multiplication satisfying
\begin{equation}\label{eq:1}
\bar{x}\bar{y}=\overline{xy}\tau(x,y)
\quad\text{and}\quad
a\bar{x}=\bar{x}a^{\sigma(x)},
\end{equation}
for all $x,y\in G$ and $a\in R$, where $\tau\colon G\times G
\to U(R)$ and $\sigma\colon G\to\Aut(R)$ are maps and $U(R)$
denotes the group of units of $R$. We can always assume that
the unity element of $K\ast G$ is $\bar{1}$.
The set $\{\bar{x}a : a\in U(R), x\in G\}$
is a subgroup of the group of units of $R\ast G$, and its
elements are called \emph{trivial units} of $R\ast G$.
(We refer to \cite{dP89} for general facts about
crossed products.)

Since we shall be looking at involutions on crossed
products, the following fact will come in handy.

\begin{lemma}\label{le:oppcp}
  Let $R$ be a ring, let $G$ be a group and let $R\ast G$ be a
  crossed product of $G$ over $R$. Then $(R\ast G)^{\op}$ is
  a crossed product of the opposite group $G^{\op}$ over the
  opposite ring $R^{\op}$.
\end{lemma}

\begin{proof}
  By definition, $R\ast G$ is a free right $R$-module
  with basis $\overline{G}$. It follows that $R\ast G$ is also
  a free left $R$-module with basis $\overline{G}$ and, hence, $(R\ast G)^{\op}$
  is a free right $R^{\op}$-module with basis $\overline{G^{\op}}$.
  Further, if $\sigma\colon G\to\Aut(R)$ and $\tau\colon G\times G\to U(R)$
  are the maps defining the multiplication in $R\ast G$, then it is easy
  to see that $(R\ast G)^{\op}=R^{\op}\ast G^{\op}$, with maps
  $$
  \begin{array}{rcl}
  G^{\op} & \longrightarrow & \Aut(R^{\op})\\
  x & \longmapsto & \sigma(x)^{-1}
  \end{array}
  \quad\text{and}\quad
  \begin{array}{rcl}
  G^{\op}\times G^{\op} & \longrightarrow & U(R^{\op})\\
  (x,y) & \longmapsto & \tau(y,x)^{\sigma(yx)^{-1}}.
  \end{array}
  $$

\end{proof}

The following is a remark on crossed products embeddable in
division rings that will be used in this section
and in the next one.

\begin{remark}\label{rem:1}
  Let $R\ast G$ be a crossed product of the group $G$ over the
  ring $R$. Suppose that there exists a division ring $D$ and a ring embedding
  $R\ast G\hookrightarrow D$. For each subgroup $N$ of
  $G$, there is a natural ring homomorphism $R\ast N\hookrightarrow R\ast G$.
  Let $D(R,N)$ denote the subdivision ring of $D$ generated
  by $R\ast N$. Then, clearly, $D(R,N) = \cup_{i\geq 0} D_i$, where
  $D_0=R\ast N$ and for each $i\geq 1$, $D_i$ is the subring of $D$ generated
  by $D_{i-1}$ and all the inverses of the nonzero elements of
  $D_{i-1}$. Further, if $N$ is a normal subgroup of a subgroup $H$ of $G$ and
  $t\in H$, then the inner automorphism $\sigma\colon D\to D$
  defined by $\sigma(f)=\bar{t}^{-1}f\bar{t}$, for all $f\in D$,
  restricts to an automorphism of $D(R,N)$. This is so, because
  $\sigma$ leaves $D_0$ invariant and, hence, $\sigma(D_i)\subseteq D_i$,
  for all $i\geq 0$.
\end{remark}

\subsection{Hughes-free embeddings}\label{ssec:li}

A group $G$ is \emph{locally indicable} if for every nontrivial
finitely generated subgroup $H$ of $G$ there exists a normal
subgroup $N$ of $H$ such that $H/N$ is infinite cyclic.
Clearly, locally indicable groups are torsion-free.

Subgroups of locally indicable groups are locally indicable.
Moreover, extensions of locally indicable groups are locally indicable
(cf.~\cite{gH40}) and orderable groups are locally indicable (by \cite{fL43});
it follows that poly-orderable groups are locally indicable.

Let $G$ be a group and let $\delta$ be an ordinal. A \emph{subnormal
series} of $G$ (indexed on $\delta$) is a chain of subgroups
$\{G_\gamma\}_{\gamma\leq\delta}$ of $G$ such that
\begin{enumerate}[(i)]
\item $G_0=\{1\}$, $G_\delta=G$,
\item $G_\gamma$ is a normal subgroup of $G_{\gamma+1}$, for all $\gamma<\delta$, and
\item for each limit ordinal $\gamma'\leq\delta$, 
  $G_{\gamma'}=\bigcup_{\gamma<\gamma'}G_\gamma.$
\end{enumerate}

Let $\mathcal{X}$ be a class of groups closed under isomorphisms. We
say that $G$ is a \emph{$\delta$-poly-$\mathcal{X}$} group if there
exists a subnormal series $\{G_\gamma\}_{\gamma\leq\delta}$ of $G$
such that for all $\gamma<\delta$ the quotient
$G_{\gamma+1}/G_\gamma\in\mathcal{X}$.

Given an ordinal $\delta$, we shall consider $\delta$-poly-orderable
groups. Note that such a group is locally indicable,
because orderable groups are locally indicable and extensions of
locally indicable groups are locally indicable. Since local
indicability is a local property of a group, it follows by
transfinite induction that a $\delta$-poly-orderable group is
locally indicable.

\medskip

Let $G$ be a locally indicable group, let $K$ be a division
ring and let $K\ast G$ be a crossed product of $G$ over $K$.
Suppose that there exists a ring embedding $K\ast G\hookrightarrow D$
of $K\ast G$ into a division ring $D$. We say that the embedding
$K\ast G\hookrightarrow D$ is \emph{Hughes-free} if for each
nontrivial finitely generated subgroup $H$ of $G$, for each
normal subgroup $N$ of $H$ for which $H/N$ is infinite cyclic,
and for each $t\in H$ such that $H/N=\langle tN\rangle$, the
set $\{\bar{t}^n : n\geq 0\}$ is right $D(K,N)$-linearly
independent inside $D$. If, moreover, $D=D(K,G)$, we say that $D$ is
a \emph{Hughes-free division ring of fractions} of $K\ast G$.

If a crossed product $K\ast G$ of a locally
indicable group $G$ over a division ring $K$ is an
Ore domain, then clearly any embedding of $K\ast G$ into a division
ring will be Hughes-free.

If $G$ is an ordered group, $K$ is a division ring and $K\ast G$ is
a crossed product of $G$ over $K$, then $K\ast G$ can be
embedded in the division ring $K\ast((G))$ of Malcev-Neumann
series (\cite{aM48, bN49}), whose elements are of the form
$$
f=\sum_{x\in G}\bar{x}a_x,
$$
with $a_x\in K$ and $\supp(f)=\{x\in G : a_x\ne 0\}$ a
well-ordered subset of $G$, where multiplication satisfies
the laws \eqref{eq:1}. It is known that the embedding 
$K\ast G\hookrightarrow K\ast ((G))$ is Hughes-free \cite{iH70}.

Hughes-free division rings of fractions are unique:

\begin{theorem}[{\cite[Theorem]{iH70}}]\label{teo:hughesI}
  Let $K$ be a division ring, let $G$ be a locally indicable group and
  let $K\ast G$ be a crossed product. Suppose that $D_1$ and $D_2$
  are Hughes-free division rings of fractions of $K\ast G$.
  Then there exists a unique ring isomorphism $\varphi\colon D_1\to D_2$ such that
  $\varphi(\alpha)=\alpha$, for all $\alpha\in K\ast G$. \qed
\end{theorem}

A proof of Theorem~\ref{teo:hughesI} can be found in \cite{DHS04}.

\medskip

We recall that a locally indicable group $\Gamma$ is said to be
\emph{Hughes-free embeddable} if every crossed product $F\ast\Gamma$
of $\Gamma$ over a division ring $F$ has a Hughes-free division ring
of fractions. For example, torsion-free nilpotent groups, or more
generally orderable groups, are Hughes-free embeddable.

It was proved in \cite{iH72} that extensions of Hughes-free
embeddable groups are Hughes-free embeddable. Thus poly-orderable
groups are Hughes-free embeddable.

\begin{lemma} Let $\mathcal{X}$ be the class of Hughes-free
  embeddable groups.
  Let $\delta$ be an ordinal and let $G$ be a $\delta$-poly-$\mathcal{X}$
  group. Then $G$ is a Hughes-free embeddable group. Hence a
  $\delta$-poly-orderable group is Hughes-free embeddable.
\end{lemma}

\begin{proof}
  Let $\{G_\gamma\}_{\gamma\leq\delta}$ be a subnormal series of $G$ such that
  $G_{\gamma+1}/G_\gamma$ is Hughes-free embeddable for each $\gamma<\delta$. We shall
  show that $G_\gamma$ is Hughes-free embeddable for all $\gamma\leq\delta$.

  If $\gamma=1$, then clearly $G_\gamma$ is Hughes-free embeddable.

  Let $\gamma>1$ and suppose that the result holds for all ordinals $\rho<\gamma$.
  If $\gamma=\rho+1$ for some ordinal $\rho$, then $G_\gamma$ is the extension of the
  Hughes-free embeddable groups $G_\rho$ and $G_\gamma/G_\rho$. By \cite{iH72},
  $G_\gamma$ is Hughes-free embeddable.

  Finally, suppose that $\gamma$ is a limit ordinal. Let $K$ be a division ring and
  let $K\ast G_\gamma$ be a crossed product. If $\varepsilon<\rho<\gamma$, then it
  follows from Theorem~\ref{teo:hughesI} that the following diagram is commutative
  $$
  \xymatrix{
  K\ast G_\varepsilon\ar@{^{(}->}[r]\ar@{^{(}->}[d] & K\ast G_\rho\ar@{^{(}->}[d]  \\
  D_\varepsilon\ar@{^{(}->}[r] & D_\rho
  }
  $$
  where $D_\varepsilon$ and $D_\rho$ are the Hughes-free division ring of fractions
  of $K\ast G_\varepsilon$ and $K\ast G_\rho$, respectively. Let
  $D_\gamma=\varinjlim_{\rho<\gamma} D_\rho$. Then $D_\gamma$ is a division ring that
  contains $K\ast G_\gamma=\varinjlim_{\rho<\gamma} K\ast G_\rho$ and $D_\rho$ for all $\rho<\gamma$.
  Let $H$ be a finitely generated subgroup of $G_\gamma$. Then $H$ is contained in some
  $G_\rho$ for some $\rho<\gamma$. Thus $K\ast H$ is contained in $K\ast G_\rho$ and $D_\rho$,
  which proves the linear independence condition that characterizes Hughes-freeness,
  because $K\ast G_\rho \hookrightarrow D_\rho$ is Hughes-free. Since $H$, $K$ and $K\ast G_\gamma$
  were arbitrary, we get that $G_\gamma$ is Hughes-free embeddable.

  The second statement follows because orderable groups are
  Hughes-free embeddable.
\end{proof}

\subsection{Extending involutions}\label{ssec:extinv}

We shall show that involutions of a crossed product of a locally
indicable group by a division ring extend to involutions on a Hughes-free
division ring of fractions. In order to prove this we shall also
present a version of Corollary~7.3 from \cite{DHS04}. We start with two lemmas
which will also be used in the proof of Theorem~\ref{teo:extension},
the main result in this section.

\begin{lemma}\label{le:iso}
  Let $K$ and $K'$ be division rings, let $G$ and $G'$
  be groups and let $K\ast G$ and $K'\ast G'$ be crossed
  products. Suppose that $K\ast G$ and $K'\ast G'$ have only trivial
  units. Then, if $\phi\colon K\ast G\to K'\ast G'$ is a ring isomorphism,
  we have
  \begin{enumerate}[(i)]
  \item $\phi(K)=K'$,\label{item1}
  \item $\phi(K^{\times}G)=(K')^{\times}G'$, and \label{item2}
  \item $\phi$ induces a group isomorphism $G\to G'$. \label{item3}
  \end{enumerate}
\end{lemma}

\begin{proof}
  Let $x\in G$ be a fixed element. Since $K'\ast G'$ has only trivial units,
  $\phi(\bar{x})=\bar{y}b$, for some $y\in G'$ and some $b\in (K')^{\times}$.
  We claim that $\phi(\bar{x}K)=\bar{y}K'$. Indeed, given $a\in K\setminus\{0,1\}$,
  we have $\phi(\bar{x}a)=\phi\bigl(\bar{x}(a-1)\bigr)+\bar{y}b$. Since $\bar{x}(a-1)$ is
  a unit, $\phi\bigl(\bar{x}(a-1)\bigr)=\bar{z}c$ for some $z\in G'$
  and some $c\in (K')^{\times}$. But $\phi(\bar{x}a)$ is also a (trivial) unit.
  Hence $z=y$ and we have $\phi(\bar{x}K)\subseteq \bar{y}K'$. Using that
  $\phi^{-1}(\bar{y}b)=\bar{x}$, one proves similarly that
  $\phi^{-1}(\bar{y}K')\subseteq \bar{x}K$.

  The above proves \ref{item2} and \ref{item1}, by setting $x=1$.
  To obtain \ref{item3}, note that the surjective group homomorphism
  $K^{\times}G \to (K')^{\times}G'/(K')^{\times}$,
  obtained by composing the isomorphism $K^{\times}G\to (K')^{\times}G'$,
  given by the restriction of $\phi$, with the canonical surjection
  $(K')^{\times}G'\to (K')^{\times}G'/(K')^{\times}$, has kernel $K^{\times}$.
  Since $G\cong K^{\times}G/K^{\times}$ and $G'\cong (K')^{\times}G'/(K')^{\times}$,
  \ref{item3} follows. (Explicitly, the isomorphism $G\to G'$ is
  given by $x\mapsto y$, where $\supp\big(\phi(\bar{x})\big)=\{y\}$.)
\end{proof}

\begin{lemma}\label{le:isohf}
  Let $K$ and $K'$ be division rings, let $G$ and $G'$ be locally indicable
  groups and let $K\ast G$ and $K'\ast G'$ be crossed products. Suppose
  that $K\ast G$ and $K'\ast G'$ are isomorphic as rings and that $K'\ast G'$
  has a Hughes-free division ring of fractions $D$. Then
  $D$ is also a Hughes-free division ring of fractions for $K\ast G$.
\end{lemma}

\begin{proof}
  Denote by $\phi\colon K\ast G\to K'\ast G'$ an isomorphism. Then $K\ast G$
  embeds into $D$ via $\phi$ and $D$ is generated by its image. We are left to prove
  that this embedding is Hughes-free.

  That $K\ast G$ and $K'\ast G'$ have only trivial units was proved in \cite{gH40}, 
  since the grading arguments of his
  Section~4 apply to crossed products.
  By Lemma~\ref{le:iso}, $\phi(K)=K'$ and there is a group isomorphism
  $\eta\colon G\to G'$ induced by $\phi$. For a subgroup $N$ of
  $G$, we then have $\phi(K\ast N)=K'\ast\eta(N)$. This implies that
  the subdivision ring of $D$ generated by (the image of) $K\ast N$
  is $D(K',\eta(N))$.

  Let $H$ be a nontrivial finitely
  generated subgroup of $G$, let $N$ be a normal subgroup of $H$ such
  that $H/N$ is infinite cyclic and pick $t\in H$ such that $H/N=\langle
  tN\rangle$. We must show that the powers of $\phi(\bar{t})$ are right linearly
  independent over $D(K',\eta(N))$. Indeed, let $d_0,\dotsc, d_n\in
  D(K',\eta(N))$ be such that $d_0+\phi(\bar{t})d_1+\dotsb+\phi(\bar{t})^nd_n=0$.
  By definition of $\eta$, for each $i=1,\dotsc,n$, there exists $b_i\in (K')^{\times}$
  such that $\phi(\bar{t})^i = \overline{\eta(t)}^ib_i$. So, we have
  $d_0+\overline{\eta(t)}b_1d_1+\dotsb+\overline{\eta(t)}^nb_nd_n=0$.
  But $\eta(t)\eta(N)$ generates $\eta(H)/\eta(N)$, which is infinite
  cyclic; since $\eta(H)$ is a nontrivial finitely generated subgroup
  of $G'$ and $K'\ast G'\hookrightarrow D$ is Hughes-free, it follows
  that $d_0=0$ and $b_id_i=0$ for all $i=1,\dotsc, n$. But, then,
  $d_i=0$ for all $i=1,\dotsc, n$. So the embedding of $K\ast G$ into $D$ through
  $\phi$ is Hughes-free.
\end{proof}

Let $G$ be a locally indicable group, let $K$ be a division ring and
let $D$ be a Hughes-free division ring of fractions of a crossed
product $K\ast G$. Given a ring automorphism $\phi$ of $K\ast G$,
the embedding of $K\ast G$ into $D$ through $\phi$ is a Hughes-free
division ring of fractions of $K\ast G$, by Lemma~\ref{le:isohf}. If
follows from Theorem~\ref{teo:hughesI} that $\phi$ extends to a ring
isomorphism of $D$. This proves the following result.

\begin{proposition}\label{cor:teohughes}
  Let $K$ be a division ring, let $G$ be a locally indicable group,
  and let $D$ be a Hughes-free division ring of fractions of
  a crossed product $K\ast G$. Then there is a natural injective group
  homomorphism $\Aut(K\ast G)\rightarrow \Aut(D)$. More precisely,
  if $\phi$ is a ring automorphism of $K\ast G$, and $\iota\colon K\ast G
  \rightarrow D$ denotes the Hughes-free embedding of $K\ast G$ into $G$, then
  there exists a unique ring automorphism $\psi$ of $D$ such that
  $\psi\iota=\iota\phi$. \qed
\end{proposition}

At first sight, it could seem that Proposition~\ref{cor:teohughes}
is a more general version of \cite[Corollary~7.3]{DHS04}. We remark
that Lemma~\ref{le:iso} together with the fact that $K\ast G$ has
only trivial units imply that any automorphism of $K\ast G$ sends
$K$ to $K$, and thus Proposition~\ref{cor:teohughes} and
\cite[Corollary~7.3]{DHS04} are the same result.

The next lemma allows us to apply Theorem~\ref{teo:hughesI} to involutions.

\begin{lemma}\label{le:opp}
  Let $K$ be a division ring, let $G$ be a locally indicable group, and
  let $K\ast G$ be a crossed product of $G$ over $K$. Let $D$
  be a division ring. Then an embedding $K\ast G\hookrightarrow D$
  is Hughes-free if and only if $(K\ast G)^{\op}\hookrightarrow D^{\op}$
  is Hughes-free.
\end{lemma}

\begin{proof}
  By Lemma~\ref{le:oppcp}, $(K\ast G)^{\op}=K^{\op}\ast G^{\op}$.
  It follows that if $N$ is a subgroup of $G$, then
  $D(K,N)^{\op}=D^{\op}(K^{\op},N^{\op})$. Given a finitely generated
  subgroup $H^{\op}$ of $G^{\op}$ and a normal subgroup $N^{\op}$ of
  $H^{\op}$ such that $H^{\op}/N^{\op}$ is infinite cyclic, pick
  $t\in H$ such that $H^{\op}/N^{\op}=\langle N^{\op}t\rangle$. We must
  prove that $\{\bar{t}^n : n\geq 0\}$ is a right $D(K,N)^{\op}$-linearly
  independent subset of $D^{\op}$, or equivalently, that it is
  a left $D(K,N)$-linearly independent subset of $D$. So suppose that
  $d_0,d_1,\dotsc,d_n\in D(K,N)$ are such that
  $$
  0=d_0+d_1\bar{t}+\dotsb+d_n\bar{t}^n=
  d_0+\bar{t}(\bar{t}^{-1}d_1\bar{t})+\dotsb+\bar{t}^n(\bar{t}^{-n}d_n\bar{t}^n).
  $$
  By Remark~\ref{rem:1}, $\bar{t}^{-i}d_i\bar{t}\in D(K,N)$ for all $i=1,\dotsc,n$.
  Then, Hughes-freeness of $K\ast G\hookrightarrow D$ implies that
  $\bar{t}^{-1}d_i\bar{t}=0$, and therefore $d_i=0$, for all $i=0,\dotsc,n$.
\end{proof}

\begin{theorem}\label{teo:extension}
  Let $K$ be a division ring, let $G$ be a locally indicable group
  and let $K\ast G$ be a crossed product. Suppose that $K\ast G$
  has a Hughes-free division ring of fractions $D$. Then any
  involution on $K\ast G$ extends to a unique involution on $D$.
\end{theorem}

\begin{proof}
  Let ${}^{\star}\colon K\ast G\to K\ast G$ be an involution on
  $K\ast G$. Then the map
  $$
  \begin{array}{rcl}
  \phi\colon K\ast G & \longrightarrow & (K\ast G)^{\op}\\
  \alpha & \longmapsto & \alpha^{\star}
  \end{array}
  $$
  is a ring isomorphism. By Lemma~\ref{le:oppcp}, $(K\ast G)^{\op}
  =K^{\op}\ast G^{\op}$ and, by Lemma~\ref{le:opp},
  $K^{\op}\ast G^{\op}\hookrightarrow D^{\op}$ is a Hughes-free
  embedding. But then Lemma~\ref{le:isohf} implies that
  the embedding of $K\ast G$ into $D^{\op}$ through $\phi$ is
  Hughes-free. So it follows from Theorem~\ref{teo:hughesI}, that
  there exists a ring isomorphism $\Phi\colon D\to D^{\op}$
  extending $\phi$. Hence, the map $D\to D$, defined by
  $\zeta\mapsto \Phi(\zeta)$, for all $\zeta\in D$, is an involution
  on $D$ extending ${}^{\star}$.
\end{proof}

\section{The general case: free symmetric generators}\label{sec:gen-free}

Here we go back to our original problem, that of searching for
free symmetric elements in a division ring generated by a
group over a field and endowed with an involution.

Let $\field$ be a field and let $G$ be a locally indicable group.
Suppose that the group ring $\field[G]$ has a Hughes-free division
ring of fractions $D$. By Theorem~\ref{teo:extension}, $D$ has an
involution ${}^{\star}$ extending the  canonical
$\field$-involution on $\field[G]$, which will also
be called the canonical involution on $D$.

Our main result, Theorem~\ref{teo:freegen}, is stated for
$\delta$-poly-orderable groups. To prove it, we need two
propositions which are valid for more general locally indicable
groups.

\begin{proposition}\label{prop:freemon}
  Let $\field$ be a field and let $G$ be a locally indicable group.
  Suppose that the group algebra $\field[G]$ has a Hughes-free division ring
  of fractions $D$. If $G$ contains a free monoid of rank $2$, then $D$
  contains a free group $\field$-algebra of rank $2$ freely generated by
  symmetric elements with respect to the canonical involution on $D$.
\end{proposition}

\begin{proof}
  Here we simplify the notation introduced in Section~\ref{sec:gen-ext}
  and given a subgroup $N$ of $G$ write $D(N)$ for the division
  subring of $D$ generated by $\field[N]$.

  Let $x,y\in G$ be elements that freely generate a free submonoid
  of $G$ and let $H$ be the subgroup of $G$ generated by $x$ and $y$.
  We shall prove that $D(H)$ contains a free group $\field$-algebra
  freely generated by symmetric elements.

  Since $G$ is locally indicable, there exists a subgroup $N$ of $H$
  such that $H/N$ is infinite cyclic. Let $t\in H$ be such that
  $tN$ generates $H/N$. By Remark~\ref{rem:1}, there exists an
  automorphism $\sigma$ of $D(N)$ extending the inner automorphism
  of $\field[H]$ given by conjugation by $t$. Since
  $\field[G]\hookrightarrow D$ is Hughes-free, the powers of $t$
  are right linearly independent over $D(N)$. Hence,
  $D(N)[t;\sigma]$ embeds into $D(H)$. Thus, $D(H)=D(N)(t;\sigma)$.
  Now let $\nu$ denote the discrete valuation on $D(H)$ extending
  the $t$-adic valuation on $D(N)[t;\sigma]$. In other words, the
  division ring $D(H)$ has a discrete valuation $\nu$ such that
  $$
  \nu\left(\sum_{n\in\Z}t^nf_n\right) = \min\{n : f_n\ne 0\},
  $$
  for all finite sums of the form $\sum_{n\in\Z}t^nf_n$, with
  $f_n\in D(N)$.

  Since $H/N$ is infinite cyclic, there exist unique $m,n\in\Z$
  and $u,v\in N$ such that $x=t^mu$ and $y=t^nv$. So $\nu(x)=m$ and
  $\nu(y)=n$, and $m\ne0 $ or $n\ne 0$ because $H\ne N$. Suppose
  $m\ne 0$. Now, if $m<0$, substitute the pair $\{x,y\}$ for the
  pair $\{x^{-1},y^{-1}\}$ and, then, substituting $y$ for $yx^k$, for an
  appropriate choice of $k$, if necessary, we can assume that
  \begin{enumerate}[(i)]
  \item $x$ and $y$ freely generate a free monoid in $G$, and
  \item $\nu(x)>0$ and $\nu(y)>0$.
  \end{enumerate}

  Consider the following elements of $\field[H]$, $f=x+x^{-1}$
  and $g=y+y^{-1}$. They are both symmetric and, thus, their
  inverses $f^{-1},g^{-1}$ are symmetric elements in $D(H)$.
  We shall regard the elements of $D(H)$ as series via the
  embedding of $D(H)=D(N)(t;\sigma)$ into the skew Laurent
  series ring $D(N)((t;\sigma))$.

  We have
  \begin{equation}\label{eq:f-1}
  f^{-1}=(x+x^{-1})^{-1} = \big((x^2+1)x^{-1}\bigr)^{-1}
  =x(1+x^2)^{-1}.
  \end{equation}
  And, similarly, $g^{-1}=y(1+y^2)^{-1}$. We claim that
  $v(f^{-1})>0$. Indeed, since $\nu(x)>0$, then expressing
  $1+x^2$ as a polynomial in $t$, we see that $\nu(1+x^2)=0$,
  which implies $\nu(f^{-1})=\nu(x)-\nu(1+x^2)=\nu(x)>0$.
  Analogously, $\nu(g^{-1})>0$.

  First we show that $f^{-1}$ and $g^{-1}$ freely generate a free monoid
  in $D(H)$. Suppose that there exist non-negative
  integers $a_1,b_1, \dotsc, a_r,b_r$ and $c_1,d_1,\dotsc, c_s,d_s$ such that
  \begin{equation}\label{eq:fm}
  (f^{-1})^{a_1}(g^{-1})^{b_1}\dotsm(f^{-1})^{a_r}(g^{-1})^{b_r} =
  (g^{-1})^{c_1}(f^{-1})^{d_1}\dotsm(g^{-1})^{c_s}(g^{-1})^{d_s}.
  \end{equation}
  If follows from \eqref{eq:f-1} that, as elements of $D(N)((t;\sigma))$,
  $f^{-1}$ and $g^{-1}$ are
  written in a unique way as
  \begin{equation}\label{eq:sf-1}
  f^{-1} = t^{m}u + \sum_{i\geq m+1} t^i \alpha_i
  \quad\text{and}\quad
  g^{-1} = t^{n}v+\sum_{j\geq n+1} t^j \beta_j,
  \end{equation}
  with $\alpha_i,\beta_j\in D(N)$. But, then, \eqref{eq:fm}
  reads
  \begin{multline}\label{eq:tm}
  (t^mu)^{a_1}(t^nv)^{b_1}\dotsm(t^mu)^{a_r}(t^nv)^{b_r}
  + \sum_{i>m(a_1+\dotsb+a_r)+n(b_1+\dotsb+b_r)} t^i\gamma_i = \\
  (t^nv)^{c_1}(t^mu)^{d_1}\dotsm(t^nv)^{c_r}(t^mu)^{d_r}
  + \sum_{j>m(d_1+\dotsb+d_s)+n(c_1+\dotsb+c_s)} t^j\delta_j,
  \end{multline}
  with $\gamma_i,\delta_j\in D(N)$. Now \eqref{eq:tm} implies
  a relation of the form
  $$
  x^{a_1}y^{b_1}\dotsm x^{a_r}y^{b_r} =
  y^{c_1}x^{d_1}\dotsm y^{c_s}x^{d_s},
  $$
  which must be trivial.

  Now we must show that monomials in $f^{-1}$ and $g^{-1}$ are
  linearly independent over $\field$. We again regard
  $D(H)=D(N)(t;\sigma)\subseteq D(N)((t;\sigma))$, endowed with
  the $t$-adic valuation. By what we have just seen,
  given a monomial $w(f^{-1},g^{-1})$ in $f^{-1}$ and $g^{-1}$,
  there exists $h\in\ D(N)((t;\sigma))$ such that
  $$
  w(f^{-1},g^{-1}) = w(x,y) + h,
  \quad\text{with $\nu(h)>\nu\bigl(w(x,y)\bigr)=
  \nu\bigl(w(f^{-1},g^{-1})\bigr)$}.
  $$
  Suppose that the monomials in $f^{-1}$ and $g^{-1}$ are not
  linearly independent over $\field$ and pick a linear dependence
  relation
  $$
  \sum_{i=1}^r w_i(f^{-1},g^{-1})\lambda_i=0,
  $$where
  $w_i(f^{-1},g^{-1})$ are monomials in $f^{-1}$ and $g^{-1}$,
  and $\lambda_i\in\field^{\times}$, with $r$ minimal. Write $I=\{1,\dotsc, r\}$ and
  let $q=\min\bigl\{\nu\bigl(w_i(f^{-1},g^{-1})\bigr) : i\in I\bigr\}$.
  Let $J=\bigl\{i\in I : \nu\bigl(w_i(f^{-1},g^{-1})\bigr)=q\bigr\}$.
  For each $j\in J$, let $h_j\in D(N)((t;\sigma))$ be such that
  $w_j(f^{-1},g^{-1}) = w_j(x,y) + h_j$ and $\nu(h_j)>q$. We have
  \begin{equation*}\begin{split}
  0 & = \sum_{i\in I}w_i(f^{-1},g^{-1})\lambda_i
      = \sum_{j\in J}w_j(f^{-1},g^{-1})\lambda_j
        + \sum_{i\in I\setminus J}w_i(f^{-1},g^{-1})\lambda_i \\
    & = \sum_{j\in J} w_j(x,y)\lambda_j + \sum_{j\in J}h_j\lambda_j
        + \sum_{i\in I\setminus J}w_i(f^{-1},g^{-1})\lambda_i.
  \end{split}\end{equation*}
  Now, $\nu\left(\sum_{j\in J}h_j\lambda_j + \sum_{i\in I\setminus J}
  w_i(f^{-1},g^{-1})\lambda_i\right)> q$. So we must have
  $\sum_{j\in J}w_j(x,y)\lambda_j =0$.
  But since $w_j(x,y)\in H$ and the elements of $H$ are linearly
  independent over $\field$, it follows that $\lambda_j=0$ for
  all $j\in J$. If $J\ne I$, we would have a nontrivial $\field$-linear dependence relation
  among monomials in $f^{-1}$ and $g^{-1}$ with less than $r$ terms,
  which is impossible. Thus, $\lambda_i=0$ for all $i\in I$.

  Finally, it follows from \cite[Corollary~1 on p.~524]{aL84}
  that $1+f^{-1}$ and $1+g^{-1}$, which are symmetric elements of
  $D$, freely generate a free group $\field$-subalgebra of $D(H)$
  and, hence, of $D$.
\end{proof}

It is worth mentioning at this point that a locally indicable
group with at least one right order which is not of Conrad type
contains a free monoid of rank $2$ (cf.~\cite[Corollary~2.6]{RR02}).

\medskip

We shall need the following result from I.~Hughes which, although
not explicitly stated there, can be extracted from \cite{iH72} or
\cite{jS08}.

\begin{lemma}\label{le:hughesII}
  Let $G$ be a locally indicable group with a normal subgroup $N$ such
  that $G/N$ is locally indicable. Let $K$ be a division ring and let
  $K\ast G$ be a crossed product. If $K\ast N$ has a Hughes-free division ring
  of fractions $D$ and $G/N$ is Hughes-free embeddable, then $K\ast G$ has
  a Hughes-free division ring of fractions. More precisely, regarding
  $K\ast G = (K\ast N)\ast(G/N) \hookrightarrow D\ast(G/N)$, if $E$ is
  a Hughes-free division ring of fractions of $D\ast(G/N)$, then
  $E$ is a Hughes-free division ring of fractions of $K\ast G$. \qed
\end{lemma}

In what follows, given a field 
$\field$ and a locally indicable group $G$ if the group algebra
$\field[G]$ has a Hughes-free division ring of fractions,
this division ring will be denoted by $\field(G)$. It follows
from Theorem~\ref{teo:extension} that the canonical involution
of $\field[G]$ extends to an involution on $\field(G)$.  
The next result is a more general version of
\cite[Proposition~3.4]{jSpp}.

\begin{proposition}\label{prop:nilpotent}
  Let $A$ be a locally indicable group generated by two elements
  and let $B$ be normal subgroup of $A$ such that $A/B$ is a torsion-free nilpotent
  group of class $2$. Let $\field$ be a field
  and suppose that the group algebra $\field[B]$ has a Hughes-free
  division ring of fractions. Then $\field[A]$ has a Hughes-free
  division ring of fractions $\field(A)$ which contains
  a free group $\field$-algebra of rank $2$ freely generated by symmetric
  elements with respect to the canonical involution.
\end{proposition}

\begin{proof}
  Denote $\Gamma=A/B$. Since $\Gamma$ is a torsion-free nilpotent
  group, it is orderable. Let $<$ denote an order in $\Gamma$. We shall
  fix a transversal of $B$ in $A$ in the following way, for
  each $\gamma\in \Gamma$, with $1<\gamma$, choose $x_{\gamma}\in A$
  such that $Bx_{\gamma}=\gamma$ and take $x_{\gamma}^{-1}$ to be
  the representative of $\gamma^{-1}$; set $x_{1}=1$. Now regard
  $\field[A]=\field[B]\ast\Gamma$. With this choice of a transversal,
  in $\field[B]\ast\Gamma$ we have $\bar{\gamma}^{\star} = \overline{\gamma^{-1}}$,
  for all $\gamma\in\Gamma$, where ${}^{\star}$ stands for the canonical
  involution on  $\field[A]$.
  
  By Lemma~\ref{le:hughesII}, $\field[A]$ has a Hughes-free division
  ring of fractions $\field(A)$, which coincides with the
  division ring of fractions of $\field(B)\ast\Gamma$ inside 
  $\field(B)\ast((\Gamma))$. Thus, we can consider the following commutative diagram
  $$\xymatrix{%
  \field[B]\ast\Gamma=\field[A]\ar@{^{(}->}[r]\ar@{^{(}->}[d] & \field(A)\ar@{^{(}->}[d]  \\
  \field[B]\ast((\Gamma))\ar@{^{(}->}[r] & \field(B)\ast((\Gamma))
  }$$

  By Theorem~\ref{teo:1}, there exist $U,V\in\field(\Gamma)$ such
  that $U^{\star}=U$, $V^{\star}=V$, and $U$ and $V$ freely generate
  a free group $\field$-subalgebra of rank $2$. Then, clearly,
  $X=UV$ and $X^{\star}=VU$ freely generate a free group $\field$-algebra
  in $\field(\Gamma)$ of rank $2$. Let
  $\varepsilon\colon\field[B]\to\field$ denote the augmentation map and
  consider the homomorphism
  $$
  \begin{array}{rcl}
  \varphi\colon\field[B]\ast((\Gamma)) & \longrightarrow & \field((\Gamma))\\
  \sum_{\gamma\in\Gamma}\bar{\gamma}h_{\gamma} & \longmapsto &
  \sum_{\gamma\in\Gamma} \gamma\varepsilon(h_{\gamma})
  \end{array}.
  $$
  We shall regard $\field(\Gamma)$ embedded into $\field((\Gamma))$ and shall
  show that with an appropriate choice of a preimage of $X$, we can
  pullback the free group $\field$-algebra from $\field(\Gamma)$ to
  $\field(A)$ through $\varphi$.

  Choose $E,F\in\field[\Gamma]$ such that $X=EF^{-1}$.
  If $E=\sum_{\gamma\in\Gamma}\gamma e_{\gamma}$ and $F=\sum_{\gamma\in\Gamma}
  \gamma f_{\gamma}$, with $e_{\gamma},f_{\gamma}\in\field$, let
  $\hat{E}=\sum_{\gamma\in\Gamma}\bar{\gamma}e_{\gamma}\in \field[B]\ast\Gamma=\field[A]$
  and $\hat{F}=\sum_{\gamma\in\Gamma}\bar{\gamma}f_{\gamma}\in \field[B]\ast\Gamma=\field[A]$.
  Then $\varphi(\hat{E})=E$, $\varphi(\hat{F})=F$, and
  $$
  \varphi(\hat{E}^{\star})=\varphi\left(\sum_{\gamma\in\Gamma}
  \bar{\gamma}^{\star}e_{\gamma}\right) = \sum_{\gamma\in\Gamma}
  \varphi\bigl(\overline{\gamma^{-1}}\bigr)e_{\gamma} =
  \sum_{\gamma\in\Gamma} \gamma^{-1}e_{\gamma}=E^{\star}.
  $$
  Similarly, $\varphi(\hat{F}^{\star})=F^{\star}$.

  Now, since the coefficients of $\hat{F}$ lie $\field$, both
  $\hat{F}^{-1}$ and $(\hat{F}^{\star})^{-1}$ lie in $\field[B]\ast((\Gamma))$,
  and we have $\varphi(\hat{E}\hat{F}^{-1})=
  EF^{-1}=X$ and $\varphi\bigl((\hat{F}^{\star})^{-1}\hat{E}^{\star}\bigr) =
  (F^{\star})^{-1}E^{\star}=(EF^{-1})^{\star}=X^{\star}$. Letting $W=\hat{E}\hat{F}^{-1}$,
  we get $W\in\field(A)$ with the property that $W$ and $W^{\star}$ freely generate a
  free group $\field$-subalgebra of $\field(A)$. Thus, with respect to the canonical
  involution,  $WW^{\star}$ and $W^{\star}W$ are symmetric elements
  of $\field(A)$,  which freely generate a free group
  $\field$-subalgebra.
\end{proof}

We are ready to tackle our main objective. But first we recall a result of 
Longobardi, Maj and Rhemtulla, which will be used in the proof.

\begin{lemma}[{\cite[Corollary~3]{LMR95}}]\label{le:solvable}
  Let $G$ be a finitely generated group with no free submonoid. Then
  for every positive integer $n$, the nth derived subgroup $G^{(n)}$
  of $G$ is finitely generated. In particular if $G$ is solvable, 
  then it is polycyclic. \qed
\end{lemma}

\medskip

\begin{theorem}~\label{teo:freegen}
  Let $\delta$ be an ordinal, let $G$ be a $\delta$-poly-orderable group,
  let $\field$ be a field and let $\field(G)$ be the Hughes-free
  division ring of fractions of the group algebra $\field[G]$.
  Then the following are equivalent:
  \begin{enumerate}[(1)]
  \item $\field(G)$ contains a free group $\field$-algebra of rank $2$ freely 
    generated by symmetric elements with respect to the canonical involution.\label{c:1}
  \item $\field(G)$ is not a locally P.I. $\field$-algebra.\label{c:2}
  \item $G$ is not locally abelian-by-finite.\label{c:3}
  \end{enumerate}
\end{theorem}

\begin{proof}
  We claim that for each $x\in \field(G)$, there exists a finitely
  generated subgroup $H_x$ of $G$ such that $x\in \field(H_x)$.
  Indeed, first observe that 
  $$
  \field(G)=\bigcup_{n\geq 0} Q_n,
  $$ 
  where $Q_0=\field[G]$ and, for each $n\geq 0$, $Q_{n+1}$ is the subring 
  of $\field(G)$ generated by $Q_n$ and the inverses of its nonzero elements. 
  If $x\in Q_0$, the result is clear. Suppose by induction that the claim holds 
  for $n\geq 0$. If $x\in Q_{n+1}$, then $x=\sum x_{i1}^{\varepsilon_{i1}}
  \dotsm x_{in_i}^{\varepsilon_{in_i}}$, where $x_{ij}\in Q_n$ and
  $\varepsilon_{ij}\in \{1,-1\}$. For each $i$ and each $j$, let $H_{x_{ij}}$ be the 
  finitely generated subgroup given by the induction hypothesis. If we set 
  $H_x$ as the subgroup generated by all $H_{x_{ij}}$, then $x\in \field(H_x)$.

  From this claim it follows that for each finitely generated $\field$-subalgebra 
  $R$ of $\field(G)$ there exists a finitely generated subgroup $H$ of $G$ such 
  that $R\subseteq \field(H)$. It is well known that if $H$ is an
  abelian-by-finite group, then $\field(H)$ is of finite dimension
  over its center (see, e.g., \cite{aL78}), and thus a P.I.~algebra. So condition \ref{c:2} implies 
  condition \ref{c:3}.

  Condition \ref{c:1} implies condition \ref{c:2} because a P.I.~algebra does not contain 
  a noncommutative free subalgebra.
  
  Now, suppose that condition \ref{c:3} holds.
  If $G$ contains a free monoid, the result follows from Proposition~\ref{prop:freemon}.

  An ordered group belongs to one and only one
  of the three classes defined in \cite{jSpp}. Ordered groups belonging
  to two of these classes contain a free monoid of rank $2$. Thus,
  if, for some $\gamma<\delta$, $G_{\gamma+1}/G_\gamma$ belongs to one of
  these, then $G_{\gamma+1}/G_\gamma$ and, \textit{a fortiori},
  $G$ contain a free monoid.  On the other hand, if a nonabelian factor 
  $G_{\gamma+1}/G_\gamma$, for some $\gamma<\delta$, lies in
  the remaining class of ordered groups, then $G_{\gamma+1}/G_\gamma$ contains a
  subgroup $\mathcal{A}$ generated by two elements $\alpha,\beta$ with a normal subgroup
  $\mathcal{B}$ such that $\mathcal{A}/\mathcal{B}$ is a nonabelian torsion-free nilpotent
  group of class $2$. Pick $x,y\in G_{\gamma+1}$ such that $xG_\gamma=\alpha$ and $yG_\gamma=\beta$,
  and let $A$ be the subgroup of $G_\gamma$ generated by $x$ and $y$. Then $A$
  is a subgroup of $G_\gamma$ containing a normal subgroup $B$ such
  that $A/B$ is a nonabelian torsion-free nilpotent group of class $2$. By
  Proposition~\ref{prop:nilpotent},
  $\field(A)\subseteq \field(G)$ contains a free group $\field$-algebra of rank $2$ 
  freely generated by symmetric elements with respect to the canonical involution.

  Thus we can suppose that $G$ does not contain a free monoid and that $G_{\gamma+1}/G_{\gamma}$ 
  is abelian for all $\gamma<\delta$. Let $H$ be a finitely generated subgroup of $G$ which 
  is not abelian-by-finite. We claim that $H$ is solvable. Indeed, by Lemma~\ref{le:solvable}, 
  we know that for each positive integer $n$, the derived subgroup $H^{(n)}$ is finitely generated. Let
  $\gamma_n$ be the first ordinal such that $H^{(n)}\subseteq
  G_{\gamma_n}$. Note that $\gamma_n$ is not a limit ordinal.
  But, for each nonlimit ordinal $\gamma=\varepsilon +1$,
  since $G_{\varepsilon +1}/G_\varepsilon$ is abelian, we have that
  $[G_{\varepsilon+1}, G_{\varepsilon +1}]\subseteq
  G_{\varepsilon}$. Thus, if $H$ is not a solvable group, using the fact
  that $H^{(n)}$ is finitely generated and locally indicable, we get that
  $\{\gamma_n\}$ is a  strictly decreasing sequence of ordinals, which is impossible.
  Hence $H$ is solvable. Moreover, since $H$ does not contain a free
  monoid, \cite[Theorems~4.7 and 4.12]{jR74}
  or \cite[Theorem~1]{LMR95} imply that $H$ is nilpotent-by-finite,
  but not abelian-by-finite by hypothesis. Hence, being $H$ torsion-free, it
  contains a torsion-free nilpotent subgroup $L$ of class $2$.
  Therefore, if follows from Theorem~\ref{teo:1}, that $\field(L)\subseteq \field(G)$ contains a 
  free group $\field$-algebra of rank $2$ freely generated by symmetric elements
  with respect to the canonical involution.
\end{proof}

As a first remark on Theorem~\ref{teo:freegen}, note that it
generalizes \cite[Theorem~3.1]{jSpp}, but the proof of the latter
is more elementary since no use of \cite{iH70}, \cite{iH72} and
\cite{jR74} is made

Secondly, if $G$ is an orderable group, then condition \ref{c:3} in 
Theorem~\ref{teo:freegen} is equivalent to $G$ being not abelian, because any 
abelian-by-periodic orderable group is abelian. Indeed, suppose that $H$ is an abelian
normal subgroup of $G$ and that $G/H$ is periodic. Let $z\in H$ and let $g\in G$. There 
exists $n\geq 1$ such that $g^n\in H$. Thus $zg^n=g^nz$ and  $(zgz^{-1})^n=g^n$. 
Since roots are unique in an orderable group, we get that $zg=gz$. Because $z\in H$ and 
$g\in G$ are arbitrary, we get that $H$ is contained in the center $Z$ of $G$. 
It is well known that in an orderable group the quotient of the group by its center 
is orderable. Thus, $G/Z\cong (G/H)/(Z/H)$ is both orderable and periodic. Therefore $G=Z$, 
as desired.

\begin{corollary}
  Let $\field$ be a field and let $G$ be a torsion-free polycylic group. Then the 
  Ore division ring of fractions of the group algebra $\field[G]$ contains a free 
  group $\field$-algebra of rank $2$ freely generated by symmetric elements with 
  respect to the canonical involution if and only if $G$ is not abelian-by-finite.
\end{corollary}

\begin{proof}
  Let $D$ denote the Ore division ring of fractions of $\field[G]$.
  It is well known that $G$ is poly-\{infinite cyclic\}-by-finite. Let $H$ be a normal 
  poly-\{infinite cyclic\} (and, thus, poly-orderable) subgroup of $G$ of finite index. 
  By Theorem~\ref{teo:freegen}, $\field(H)\subseteq D$ contains a free group $\field$-algebra 
  of rank $2$ freely generated by symmetric elements with respect to the canonical involution 
  if and only if $H$ is not abelian-by-finite. If $G$ is not abelian-by-finite, 
  then $H$ is not abelian-by-finite, and the result follows. The proof of the reverse implication 
  is the same as the proof that condition \ref{c:1} implies condition \ref{c:3} in
  Theorem~\ref{teo:freegen} (in which the hypothesis on $G$ is not needed).
\end{proof}

\begin{remark}
  In view of Propositions~\ref{prop:freemon} and \ref{prop:nilpotent},
  we think it might be possible to prove that a Hughes-free division
  ring of fractions of the group algebra of a locally indicable
  group which is not locally abelian-by-finite
  will always contain a free group algebra freely generated by symmetric elements with
  respect to the canonical involution. The evidence we have for this is as follows.

  Let $G$ be a locally indicable group, let $\field$ be a field and suppose that
  the group ring $\field[G]$ has a Hughes-free division ring of fractions $D$. Thus 
  the canonical involution can be uniquely extended to $D$.

  We have proved that if there exists a finitely generated subgroup $H$ of $G$ such 
  that either
  \begin{enumerate}[(i)]
   \item $H$ contains a free monoid of rank $2$, or
   \item there exists a normal subgroup $N$ of $H$ such that $H/N$ is torsion-free 
     nilpotent of class $2$,
  \end{enumerate}
  then $D$ contains a free group $k$-algebra of rank $2$ generated by symmetric 
  elements. Incidentally, if $G$ is a $\delta$-poly-orderable group which is not 
  locally abelian-by-finite, then it contains such a finitely generated subgroup $H$. 
  Of course, if $G$ is abelian-by-finite, then $D$ cannot contain a free algebra of rank 2.

  Our methods do not apply in the remaining case: $G$ is not locally abelian-by-finite 
  and it does not contain a finitely generated subgroup $H$ satisfying (i) or (ii). 
  In this case, $G$ does not contain a free monoid and each finitely generated subgroup $H$ is 
  either abelian-by-finite or not solvable. (Indeed, if $H$ were solvable, then 
  \cite[Theorems~4.7 and 4.12]{jR74} or \cite[Theorem~1]{LMR95} would imply that $H$ was 
  nilpotent-by-finite. Being $H$ torsion free, $H$ would either be abelian-by-finite 
  or contain a torsion-free nilpotent subgroup of class 2 and Theorem~\ref{teo:1}
  would apply.)
  If $H$ is neither abelian-by-finite nor solvable, then
  the $n$-th derived subgroup $H^{(n)}$ of $H$ is finitely generated and $H/H^{(n)}$ 
  is polycyclic  for every positive integer
  $n$ by Lemma~\ref{le:solvable}. Moreover, since $H^{(n)}\neq \{1\}$ and $H$ is locally 
  indicable, $H^{(n+1)}\subseteq H^{(n)}$ and $H^{(n)}/H^{(n+1)}$ is a finitely generated
  abelian group which contains torsion-free elements for every positive integer $n$.
  Also $H/H^{(n)}$ is polycyclic with no noncommutative free monoid, and thus
  nilpotent-by-finite.
  Let $N$ be a normal subgroup of $H$ such that $H^{(n)}\subseteq N$, $N/H^{(n)}$ is nilpotent 
  and $H/N$ is finite. Since the elements of finite order in a nilpotent group form a 
  characteristic subgroup, there exists a subgroup $H^{(n)}\subseteq N^+\subseteq N$ such 
  that $N^+/H^{(n)}$ is finite, $N^+$ is normal in $H$ and $N/N^+$ torsion-free nilpotent. 
  Thus $N/N^+$ has to be abelian because $N$ is finitely generated and (ii) is not satisfied.

  We do not have examples of such a group $G$.
\end{remark}

\section{Further developments}\label{sec:prob}

In this last section, we shall point to possible generalizations
of the results obtained in this paper. We present two possible
directions towards which new research can be done.

\subsection{General involutions}

We have seen in Theorem~\ref{teo:freegen} that if $\field$ is a field and $G$ is a nonabelian ordered
group, then the division ring $\field(G)$ generated by $\field[G]$ inside
the Malcev-Neumann series ring $\field((G))$ has an involution induced by
the canonical involution in $G$. Moreover, with respect to this involution,
$\field(G)$ contains a free group $\field$-algebra of rank $2$ freely generated
by symmetric elements. A natural question that arises is whether the same
can be proved for other kinds of involutions. We shall address this question for involutions 
in $\field(G)$ which, although not being the canonical one, still are induced by
involutions on $G$. 

We restrict to the nilpotent case and, in fact, raise more
questions than answers.

By and \emph{involution} on a group $G$ we understand a map 
${}^{\star}\colon G\to G$ satisfying
\begin{enumerate}[(i)]
\item $(xy)^{\star}=y^{\star}x^{\star}$, for all $x,y\in G$, and
\item $(x^{\star})^{\star}=x$, for all $x\in G$.
\end{enumerate}

If $\field$ is a field and ${}^{\star}$ is an involution on a group
$G$, then the group algebra $\field[G]$ can be endowed with a 
$\field$-involution, still denoted by ${}^\star$, defined by
$\left(\sum_{x\in G}xa_x\right)^{\star}= \sum_{x\in G}x^{\star}a_x$,
with $x\in G$ and $a_x\in\field$, all but a finite number of which nonzero.

For instance, if $G = \langle x,y : [[x,y],x]=[[x,y],y]=1\rangle$ is the free 
nilpotent group of class $2$ generated by two elements, considered in Section~\ref{sec:nilp},
then $G$ has an involution ${}^{\star}$ satisfying
$x^{\star}=x^{-1}$ and $y^{\star}=y$. The map $x^{\dagger}=x$, $y^{\dagger}=y$
also induces an involution on $G$ (satisfying $[x,y]^{\dagger}=[x,y]^{-1}$).
A third example of a non canonical involution ${}^{\ddagger}$ on $G$ is given by
$x^{\ddagger}=y$ and $y^{\ddagger}=x$. 
These involutions extend to $\field$-involutions on $\field[G]$ and,
therefore, to $\field$-involutions on the Ore field of fractions $D$
of $\field[G]$.

Clearly, the elements 
$$
  1+y(1-y)^{-2}
  \quad\text{and}\quad
  1+ y(1-y)^{-2}x(1-x)^{-2}y(1-y)^{-2},
$$
found in Theorem~\ref{teo:1}, which freely generate a free group $\field$-algebra in
$D$, are symmetric with respect to both ${}^{\star}$ and ${}^{\dagger}$.
The involution ${}^{\ddagger}$ on $G$ seems more mysterious and our
methods do not provide an answer. We, thus, suggest the following
problem.

\begin{problem}
  Let $\field$ be a field and consider the group 
  $G=\langle x,y : [[y,x],x]=[[y,x],y]=1\rangle$. Let $D$ be the 
  field of fractions of the group algebra $\field[G]$ and let ${}^{\ddagger}$
  denote the $\field$-involution on $D$ induced by the involution on $G$
  such that $x^{\ddagger}=y$ and $y^{\ddagger}=x$. Find a pair of symmetric
  elements in $D$ that freely generate a free group $\field$-algebra.
\end{problem}

Rather, more generally, one could ask whether for every $\field$-involution
on $D$ which is induced by an involution on $G$, there will always exist
a pair of symmetric elements freely generating a free group $\field$-algebra 
inside $D$.

\subsection{Twisted involutions}

Given an involution ${}^{\star}$ on a group $G$ and a group
homomorphism $c\colon G\to\field^{\times}$ into the multiplicative group
of a field $\field$, the map
$$
\begin{array}{rcl}
  \field[G] & \longrightarrow & \field[G]\\
  \sum_{x\in G}xa_x & \longmapsto & \sum_{x\in G}x^{\star}c(x)a_x,
\end{array}
$$
where, for all $x\in G$, $a_x$ are elements of $\field$ all but a
finite number of which nonzero, is a $\field$-involution on
$\field[G]$. These kind of involutions in $\field[G]$ will be called \emph{twisted
involutions}.

One might ask whether the results in this paper hold for general
twisted involutions or, particularly, for twisted involutions
induced by the canonical involution on $G$.

We have not explored twisted involutions throughly, but some of
the results on the paper do have a version for them. For instance,
the following version of Proposition~\ref{prop:freemon} can be
proved in the same way.

\begin{proposition}
  Let $\field$ be a field, let $G$ be a locally indicable group
  and let $c\colon G\to \field^{\times}$ be a group homomorphism.
  Suppose that the group algebra $\field[G]$ has a Hughes-free division ring
  of fractions $D$. If $G$ contains a free monoid of rank $2$, then $D$
  contains a free group $\field$-algebra of rank $2$ freely generated by
  symmetric elements with respect to the $c$-twisted canonical involution on $D$.
\end{proposition}

Here we have called \emph{$c$-twisted canonical involution} the extension
to $D$ of the twisted involution induced by the canonical
involution on $G$ and $c$. For the proof we need slight modifications
of the free generators: we take $f=c(x^{-1})x+x^{-1}$ and $g=c(y^{-1})y+y^{-1}$
for the same $x$ and $y$.

\end{document}